\documentclass[10pt]{article}
\usepackage{amsfonts,amsmath,amssymb,amsthm}
\usepackage{bbm,bm}
\usepackage{enumitem}
\usepackage[mathcal,mathscr]{euscript}
\usepackage[margin=1.0in]{geometry}
\usepackage[colorlinks]{hyperref}
\usepackage{mathtools} 
\usepackage{mathrsfs}  
\usepackage[numbers]{natbib}

\numberwithin{equation}{section}
\newtheorem{theorem}{Theorem}[section]
\newtheorem{lemma}[theorem]{Lemma}
\newtheorem{proposition}[theorem]{Proposition}
\newtheorem{corollary}[theorem]{Corollary}
\theoremstyle{definition} 
\newtheorem{example}[theorem]{Example}
\newtheorem{definition}[theorem]{Definition}
\newtheorem{remark}[theorem]{Remark}
\newtheorem{assumption}[theorem]{Assumption}

\usepackage{datetime}

\def\EE{\mathbb{E}}
\def\NN{\mathbb{N}}
\def\PP{\mathbb{P}}
\def\RR{\mathbb{R}}

\def\cM{\mathcal{M}}
\def\cP{\mathcal{P}}
\def\cX{\mathcal{X}}
\def\cK{\mathcal{K}}
\def\cD{\mathcal{D}}
\def\cW{\mathcal{W}}

\def\ra{\rightarrow}
\def\eqdist{\buildrel (d) \over = }

\def\Rna{R_{\mathbf{a}}^{(n)}}
\def\lk{\mathsf{L}_{n,k}}
\def\lkz{\mathsf{L}^{\mathbf{Z}}_{n,k}}
\def\tk{\mathbb{U}_k}

\def\qnk{\mathbf{Q}_{n,k}}
\def\rnk{\mathbf{R}_{n,k}}
\def\vnk{\mathbb{V}_{n,k}}
\def\znk{\mathbf{Z}_{n,k}}

\def\newxi{\mathsf{s}}
\def\newkappa{\delta}
\def\cone{\eta}
\def\newlambda{\bar{\Lambda}}

\begin{document}

\begin{center} 
  \textbf{\Large  Large deviation principles induced  by the Stiefel manifold, and  random  multi-dimensional projections}
\vskip 6mm
\textit{Steven Soojin Kim\footnote{These results are part of this author's Ph.D. thesis \cite{skim-thesis}} and  Kavita Ramanan\footnote{Kavita Ramanan was supported in part by NSF-DMS Grants 1713032 and 1954351}}\\
\vskip 3mm
{Brown University}\\
\end{center}

\begin{center}
\textit{This article is dedicated to the memory of Elizabeth Meckes} 
\end{center}

\begin{abstract}
  For fixed positive integers $k < n$, given an $n$-dimensional random vector $X^{(n)}$, consider 
  its $k$-dimensional projection $\mathbf{a}_{n,k}X^{(n)}$, where
  $\mathbf{a}_{n,k}$ is an $n \times k$-dimensional matrix
   belonging to the
   Stiefel manifold $\mathbb{V}_{n,k}$ of orthonormal $k$-frames in $\mathbb{R}^n$.
For a class of sequences  $\{X^{(n)}\}_{n \in \NN}$  that includes 
the uniform distributions on suitably scaled $\ell_p^n$ balls, $p \in  (1,\infty]$, and
  product measures with sufficiently light tails, 
      it is shown that the sequence of projected vectors  
      $\{\mathbf{a}_{n,k}^\intercal X^{(n)}\}_{n \in \NN}$ satisfies a large deviation principle  whenever 
  the  empirical measures of the rows of $\sqrt{n} \mathbf{a}_{n,k}$ converge, as $n \rightarrow \infty$, to a probability 
  measure on $\RR^k$.
  In particular,  when $\mathbf{A}_{n,k}$ is a random matrix drawn  from the Haar measure on $\mathbb{V}_{n,k}$, 
  this is shown to imply a large deviation principle for the sequence of random projections $\{\mathbf{A}_{n,k}^\intercal X^{(n)}\}_{n \in \NN}$ in the {\em quenched} sense
  (that is, conditioned on  almost sure realizations of $\{\mathbf{A}_{n,k}\}_{n \in \NN}$). 
  Moreover, a variational formula is obtained for the rate function
  of the large deviation principle for the {\em annealed} projections $\{\mathbf{A}_{n,k}^\intercal X^{(n)}\}_{n \in \NN}$,  which is expressed in terms of  a family of quenched rate functions and
  a modified entropy term. 
  A key step in this analysis is  a large deviation principle for the sequence of empirical measures of rows of
  $\sqrt{n} \mathbf{A}_{n,k}$, $n \geq k$,  
  which   may be of independent interest.   The study of multi-dimensional random projections of high-dimensional measures is of interest in asymptotic functional analysis, convex geometry and  statistics.  Prior results on quenched large
  deviations for random projections   of $\ell_p^n$ balls have been essentially restricted to the one-dimensional setting. 
\end{abstract}



\section{Introduction}

\subsection{Background}\label{ssec-background}

The study of high-dimensional measures and their lower-dimensional projections is a central theme in  high-dimensional probability, asymptotic functional analysis and  
 convex geometry, where in the latter case  the measures of interest are distributions on 
convex bodies, which are compact, convex sets with non-empty interior  (see, e.g., \cite{Kla07a,Mec12b}).  
Multidimensional projections  of high-dimensional random vectors are  also relevant  in
statistics,  data analysis and computer science \cite{DiaFre84,JohLin84}.  
 Recent work has shown that large deviation principles (LDPs) that capture the tail behavior
  of lower-dimensional random projections can provide more interesting information about the original high-dimensonal measures
  than central limit theorem type results that capture universal phenomena of fluctuations. For example, in the case of
    $\ell_p^n$ balls,  $p \in [1,\infty)$, which are  fundamental objects in convex geometry, 
  this  was first  illustrated by 
    LDPs for  one-dimensional projections obtained in \cite{GanKimRam17,skim-thesis}, and subsequently by LDPs 
  for norms  of samples from $\ell_p^n$ balls and their multi-dimensional projections   in 
  \cite{AloGutProTha18,KabProTha19a,KimLiaRam20},  as well as corresponding 
  refined large deviation estimates   
   obtained in  \cite{LiaRam20a,Kaufmann21}.  
  LDPs of random projections of high-dimensional measures
  are broadly of two types, the terminology arising from statistical physics:  so-called ``quenched'' LDPs, where one conditions on the choice of the (sequence of) sub-spaces, bases or
  directions onto which one  projects; or  ``annealed'' LDPs, which  average over the randomness arising in the choice of the projection. 
 While most of the work described above on $\ell_p^n$ balls focused on one-dimensional LDPs (either for one-dimensional projections or norms of higher-dimensional projections), 
  in \cite{KimLiaRam20}, annealed LDPs were also established for multi-dimensional projections
  of high-dimensional measures that  satisfy  a general condition called the
  asymptotic thin shell condition. This condition was shown to be satisfied in 
  \cite{KimLiaRam20} by several classes of measures, including
   product measures with polynomial tail decay,  
  $\ell_p^n$ balls,  $p \in [1,\infty]$, and classes of Orlicz balls and  Gibbs measures.

In this article, we  establish 
quenched LDPs for multidimensional random projections of a class of high-dimensional measures
as the dimension $n$ goes to infinity.
 Quenched LDPs can often provide more geometric information
    than annealed LDPs, but their analysis is typically more difficult because one can no longer exploit
    symmetry properties of the random projection measure.  
To state our results more precisely, for $k \in \RR^n$, let $I_k$ denote the $k\times k$
 identity matrix, and for $n > k$, let 
\begin{equation}\label{eq-vnk}
\vnk := \{A \in \RR^{n\times k} : A^\intercal  A = I_k\}
\end{equation}
denote the Stiefel manifold of $k$-frames in $\RR^n$.   Observe that the set  $\mathbb{V}_{n,n}$ can be
  identified with the set ${\mathcal O}(n)$ of $n \times n$ orthogonal matrices with columns of
  norm $1$. Also, note that 
  for $k, n \in \NN$, $k < n$, any $a_{n,k} \in \vnk$ defines a linear
  projection from $n$ to $k$ dimensions.  
  Fixing a probability space $(\Omega,\mathcal{F},\PP)$, we consider sequences of random vectors $\{X^{(n)}\}_{n \in \NN}$ defined on this space 
  that satisfy a certain set of conditions (see Assumption \ref{as-quenched}), 
  which include, for example,   $X^{(n)}$  uniformly distributed on an $\ell_p^n$ ball of radius $n^{1/p}$,
  $p \geq 2$,
  or $X^{(n)}$ distributed  according to a product measure with sufficiently light tails. 
For any fixed $k \in\NN$, let $\NN_k := \{n \in \NN: n > k\}$, and 
consider the sequence of  $k$-dimensional projections 
\begin{equation}\label{eq-atx2}
\{n^{-1/2} \mathbf{a}_{n,k}^\intercal X^{(n)}	\}_{n\in\NN_k},
\end{equation}
where $\mathbf{a} := \{\mathbf{a}_{n,k}\}_{n \in \NN_k}$, with $\mathbf{a}_{n,k} \subset \mathbb{V}_{n,k}$ for each $n \in \NN_k$.
Also, let 
 \begin{equation}\label{emp_L}
  \mathsf{L}_{n,k}^{\mathbf{a}} := \frac{1}{n}\sum_{i=1}^n \delta_{\sqrt{n}\mathbf{a}_{n,k}(i,\cdot\,)}, \quad n\in\NN_k,   
 \end{equation}
 be the associated sequence of empirical  measures of the rows of $\sqrt{n} \mathbf{a}_{n,k}$.  
 Our first result, Theorem \ref{th-quenched}, shows that whenever $\{\mathsf{L}_{n,k}^{\mathbf{a}}\}_{n \in \NN_k}$
 converges in the $q_\star$-Wasserstein topology (see Definition \ref{def-Wasserstein}) to a measure $\nu$, then the sequence of
 random projections  $\{n^{-1/2} \mathbf{a}_{n,k}^\intercal X^{(n)}	\}_{n\in\NN_k},$  satisfies an LDP 
 with a rate function that we denote by ${\mathcal J}^{\sf{qu}}_\nu$. 
 In particular, this implies  a quenched LDP for the
sequence $\{\mathbf{A}_{n,k}^\intercal X^{(n)}\}_{n \in \NN_k},$ where the  random matrix 
\begin{equation*}
\mathbf{A}_{n, k} = [ \mathbf{A}_{n,k}(i,j) ]_{i=1,\dots, n; \, j=1,\dots, k} 
\end{equation*}
is sampled,  independently of $\{X^{(n)}\}_{n \in \NN}$,  from $\sigma_{n,k}$, the Haar measure on  $\vnk$
(i.e., the unique probability measure on $\vnk$ that is invariant under the group ${\mathcal O}(n)$ of orthogonal transformations).  
In \cite{KimLiaRam20}, it was shown that 
$\{\mathbf{A}_{n,k}^\intercal X^{(n)}\}_{n \in \NN_k}$ also satisfies an annealed LDP.
Our second result,  Theorem \ref{th-var}, establishes a variational formula  for the (annealed) 
rate function ${\mathcal J}^{\sf{an}}$, in terms of  the quenched rate functions ${\mathcal J}^{\sf{qu}}_\nu$.  
Along the way, for any $q \in (0,2)$, 
in  Theorem \ref{th-stldp}, we also 
establish an LDP for the random empirical measure sequence $\{ \mathsf{L}^{A}_{n,k}\}_{n \in \NN_k}$ in
the $q$-Wasserestein topology, 
which  may be of independent interest.

In the next section, we introduce some basic notation and terminology that will be used throughout, and then provide
precise  statements of our main results in Section \ref{ssec-mainresults}, with proofs presented
in Sections \ref{sec-bartlett}--\ref{sec-pfva}. 

\subsection{Notation and Terminology}
\label{subs-notat}
  
We first recall the definition of an LDP, referring to \cite{DemZeiBook} for further background on large deviations theory.
\begin{definition}\label{def-ldp}
Let $\cX$ be a topological space with Borel sigma-algebra $\mathcal{B}$. A sequence of probability measures $\{P_n\}_{n\in\NN}\subset \cP(\cX)$  is said to satisfy a \emph{large deviation principle (LDP)} in $\cX$ with rate function $I:\cX \ra [0,\infty]$ if for all $B \in \mathcal{B}$,
\begin{equation*}
  -\inf_{x\in B^\circ} I(x) \le \liminf_{n\ra\infty} \tfrac{1}{n} \log P_n(B) \le \limsup_{n\ra\infty} \tfrac{1}{n} \log P_n(B) \le -\inf_{x\in \bar{B}} I(x),
\end{equation*}
where $B^\circ$ and $\bar{B}$ denote the interior and closure of $B$, respectively. We say $I$ is a \emph{good rate function (GRF)} if it has compact level sets. Analogously, a sequence of $\cX$-valued random variables $\{\mathsf{x}_n\}_{n\in\NN}$ is said to satisfy an LDP with GRF $I$ if the sequence of their laws $\{\PP \circ \mathsf{x}_n\}_{n\in\NN}$ does.
\end{definition}
  We recall some definitions that will prove useful in our discussion of rate functions. Given a function $f:\RR^m \rightarrow [-\infty, \infty]$
      for some $m \in \NN$, 
  we let $f^*$ denote its \emph{Legendre transform}: 
   \[ f^* (t) := \sup_{ s \in \RR^m} \left[\langle s, t \rangle   -
    f (s) \right], \qquad t \in \RR^m. 
   \]
   Since we will  frequently 
   invoke the contraction principle,  Cram\'{e}r's theorem and Sanov's theorem, we refer the reader to
    Theorem 4.2.1, Corollary  6.1.6, and Section 6.2,  respectively, of  \cite{DemZeiBook}, for precise statements. 

Next, for $p\in[1,\infty]$, let $\|\cdot\|_p$ denote the $\ell_p^k$ norm on $\RR^k$. When $p = 2$, and where the context makes it clear, we will omit the subscript and simply  write $\|\cdot\|$ for the Euclidean norm.   Let $\cP(\RR^k)$ denote the space of probability measures on $\RR^k$, by default equipped with the topology of weak convergence. We will sometimes consider the following restricted subsets of probability measures:  for $q > 0$, let
\begin{equation*}
  \cP_q(\RR^k) := \left\{ \nu \in \cP(\RR^k) :   \int_{\RR^k} \|x\|^q \nu(dx) < \infty \right\}.
\end{equation*}

\begin{definition}\label{def-Wasserstein}
  For $q > 0$, we say a sequence of probability measures $\{\nu_n\}_{n\in\NN}\subset \cP_q(\RR^d)$
  converges to a limit $\nu$ with respect to the \emph{$q$-Wasserstein topology} if we have both weak convergence, denoted 
  $\nu_n\Rightarrow \nu$, as well as convergence of $q$-th moments $\int_{\RR^d} \|x\|^q \nu_n(dx) \ra \int_{\RR^d} \|x\|^q \nu(dx) $. As noted in Section 6 of \cite{Vil08}, the $q$-Wasserstein topology is metrizable through the $q$-Wasserstein metric
which we denote by $\mathcal{W}_q$.
\end{definition}

\section{Main results}\label{ssec-mainresults}

We now provide a precise statement of our results.  We start by defining the class of random vectors that we consider.
As in \cite[Definition 2.3.5]{DemZeiBook}, we define the domain of an (extended real-valued)  function $f$ defined on a Euclidean space $S$,  denoted $D_f$, as the subset of points in $S$ for which $f$ is finite; furthermore, the function $f$ is said to be \emph{essentially smooth} if: $D_{f} \neq \emptyset$; $f$ is differentiable in the interior $D_{f}^\circ$ of $D_f$; and  $f$ is ``steep" (i.e., if $D_{f}$ has a boundary $\partial D_{f}$, then $\lim_{t \rightarrow \partial D_{f}} |\nabla f (t)| = \infty$).

\begin{assumption}[quenched]\label{as-quenched} 
We impose the following conditions on the sequence of random vectors $\{X^{(n)}\}_{n\in \NN}$. 
\begin{enumerate}[label=(\roman*)]
\item \textsc{Representation:} there exists a sequence of i.i.d.\ real-valued random variables $\{\xi_j\}_{j\in \NN}$, a Borel measurable function $\mathbf{r}:\RR \ra \RR_+$, and a continuous function $\rho:\RR_+ \ra \RR_+$ such that
\begin{equation*} X^{(n)} \eqdist \xi^{(n)} \cdot \rho\left(\frac{1}{n}\sum_{i=1}^n \mathbf{r}(\xi_i) \right), \quad n\in \NN,
\end{equation*}
where $\xi^{(n)} := (\xi_1, \dots, \xi_n)$.  Let 
$\Lambda$ denote the log moment-generating function (mgf) of $(\xi_1, \mathbf{r}(\xi_1))$: 
\begin{equation}\label{lambda} \Lambda(s_1, s_2) := \log \EE\left[\exp\left(s_1 \xi_1 + s_2 \mathbf{r}(\xi_1)  \right)\right], \quad s_1 \in \RR, s_2 \in \RR. 
\end{equation}

\item \textsc{Log Moment-Generating Function (MGF):}
There exists  $q_\star > 0$ such that for  every $(s_1,s_2) \in D_\Lambda,$ 
there exists a finite constant $C_{s_2}$  (depending only on $s_2$ and not $s_1$)  such that 
\begin{equation}\label{lambd}
\Lambda(s_1,s_2) \le C_{s_2} (1 + |s_1|^{q_\star}).  
\end{equation}
Furthermore,   there  exists $T \le \infty$ such that $D_\Lambda  = \RR\times (-\infty, T)$. 

\item \textsc{Integrated Log Mgf:} For any $\nu \in \cP(\RR^k)$, the function $\Psi_\nu:\RR^{k+1} \ra \RR$ obtained
  as an integrated form of the log mgf,
  \begin{equation}\label{psi} \Psi_\nu(t_1,t_2) := \int_{\RR^k} \Lambda\left(\langle t_1, x\rangle, t_2  \right) \nu(dx), \quad t_1 \in \RR^k, t_2 \in \RR,
  \end{equation}
contains 0 in the interior of its domain, is lower semicontinuous, and is essentially smooth.

\item \textsc{Log MGF of Square:} The log mgf $\newlambda$ of $(\xi_1^2, \mathbf{r} (\xi_1))$, given by 
     \begin{equation}
       \label{cond-annealed}
       \newlambda (s_1, s_2) := \log \EE [ \exp (s_1 \xi_1^2 + s_2 \mathbf{r}(\xi_1))], \qquad 
       s_1 \in \RR, \, s_2 \in \RR,
     \end{equation} 
     is finite in a non-empty neighborhood of the origin $(0, 0)$. 
     
\item \textsc{Tail Bound:} The exponent $q_\star$ of part (ii) is bounded above, $q_\star < 2$.
 
\end{enumerate}

\end{assumption}

\begin{remark}
  \label{rem-quenched}
  The inequality \eqref{lambd} in Assumption \ref{as-quenched}(ii) implies that  
  for $t_1\in \RR^k$ and $t_2 \in (-\infty,T)$, the map $\nu \mapsto \Psi_\nu(t_1,t_2)$
    is continuous with respect to the $q_\star$-Wasserstein topology.  Further, for $t_2 > T$, 
  we have  $\Lambda (y,t_2) = \infty$ for all $y \in \RR$, and hence,
    \eqref{psi} shows that $\Psi(t_1,t_2) =  \infty$ for all $t_1 \in \RR$. 
    \end{remark}

\begin{remark}
  \label{rem-asq}
  A wide class of product measures satisfy Assumption \ref{as-quenched} with $\rho \equiv \mathbf{r} \equiv 1$; namely those that have sufficiently light tails,  in the sense of parts (iv) and (v).  Examples of  sequences of non-product measures
  satisfying Assumption \ref{as-quenched} are   $\ell_p^n$ spheres.  
  More precisely, fix $p  \in [1,\infty)$, and  for $n \in \NN$, let 
  $\mathbb{D}_{n,p}  := \{  x  \in \RR^n:  \sum_{i=1}^n  |x_i|^p  = n \}$ be the  scaled $\ell_p^n$ ball in $\RR^n$,
 let  $\mathbb{S}_{n,p} := \partial \mathbb{D}_{n,p}$ be the scaled $\ell_p^n$ sphere in $\RR^n$, let 
  $\cone_{n,p}$ be the cone measure  on $\mathbb{S}_{n,p}$: for Borel  subsets $S \subset \mathbb{S}_{n,p}$, 
  \[ \cone_{n,p} (S) := \frac{{\rm vol}_n (\{cx: x \in S, c \in [0,n^{1/p}]\})}{{\rm vol}_n (\mathbb{D}_{n,p})}, \]
  with ${\rm vol}_n$ denoting Lebesgue  measure on $\RR^n$, and let 
  $X^{(n)} = X^{(n,p)}$ be distributed according to $\cone_{n,p}$.
    Then:
   \begin{enumerate}[label=(\roman*)]
\item for $p \in [1, \infty)$, this condition follows from 
  results in \cite{SchZin90,RacRus91}, with  $\{\xi_j\}_{j\in \NN}$ being the
  i.i.d.\ sequence  with common law equal to the generalized $p$-normal distribution (namely, the probability measure on $\RR$ with density proportional to $e^{-|y|^p/p}$), $\mathbf{r}(x) = |x|^p$, and $\rho(y) = y^{-1/p}$;
  
\item for $p \in (1, \infty)$, the growth conditions on the log mgf $\Lambda$ are established in \cite[Lemma 5.7]{GanKimRam17}; further,
  $\Lambda$ is symmetric in its first argument due to the symmetry of the generalized $p$-normal distribution; 

 \item for $p \in (1, \infty)$, the conditions on the integrated log mgf are established in \cite[Lemma 5.9]{GanKimRam17};
 
 \item for $p \in [2, \infty)$, the log mgf condition is easily verified; 
  
 \item for $p \in (2, \infty)$, the precise tail bound exponent is established in \cite[Lemma 5.5]{GanKimRam17}.

\end{enumerate}
\end{remark}

We now our state our first  result, whose proof is deferred to Section \ref{sec-quenched}.  Recall the 
 $q$-Wasserstein metric $\mathcal{W}_q$ specified in Definition \ref{def-Wasserstein}.
 Also, for any $\nu \in \cP(\RR^k)$, we  let $\Psi_\nu^*$ denote the Legendre transform of $\Psi_\nu$,
 \begin{equation}\label{psistar} \Psi_\nu^*(\tau_1,\tau_2) := \sup_{t_1\in \RR^k, t_2\in \RR} \{ \langle t_1, \tau_1\rangle + t_2 \tau_2 - \Psi_\nu(t_1,t_2) \}, \quad \tau_1\in\RR^k, \tau_2\in \RR.
 \end{equation}
 Also, let $\gamma$ denote the standard Gaussian distribution on $\RR$, and $\gamma^{\otimes k}$ its $k$-fold product.

 \begin{theorem}[Quenched]\label{th-quenched}
   Fix $k \in \NN$, and suppose $\{X^{(n)}\}_{n \in \NN}$ satisfies Assumption \ref{as-quenched}(i ii, iii) with associated $q_\star > 0$ and $\Psi_\nu$.  Choose any sequence $\mathbf{a} = \{\mathbf{a}_{n,k}\}_{n\in\NN_k} \subset \vnk$ such that, for some $\nu \in \cP(\RR^k)$, 
\begin{equation}\label{eq-qwass}
  \mathcal{W}_{q_\star}(\mathsf{L}_{n,k}^{\mathbf{a}}, \nu) \to 0 
\end{equation}
where $\mathsf{L}_{n,k}^{\mathbf{a}} \in \cP(\RR^k)$ is the  empirical measure defined in \eqref{emp_L}. Then, the following claims hold: 
\begin{enumerate}[label=(\roman*)]
\item The sequence $\{n^{-1/2} \mathbf{a}_{n,k}^\intercal  X^{(n)}\}_{n\in\NN_k}$ satisfies an LDP in $\RR^k$  with GRF $\mathcal{J}_\nu^{\sf{qu}}:\RR^k \ra [0,\infty]$ defined by 
\begin{equation}
\label{def-jqu}
  \mathcal{J}_\nu^{\sf{qu}}(x) := \inf_{\tau \in \RR_+} \Psi_\nu^*\left(\tfrac{x}{\rho(\tau)},\tau\right) ,\quad x\in \RR^k. 
\end{equation}

\item If $\sigma$ is any probability measure on $\mathbb{S} := \otimes_{n > k} \mathbb{V}_{n,k}$ whose  
  $n$-th marginal coincides with the Haar measure $\sigma_{n,k}$,  then for $\sigma$-a.e. $\mathbf{a}  = \{\mathbf{a}_{n,k}\}_{n \in \NN_k} \in \mathbb{S}$,
   the sequence
  $\{n^{-1/2} \mathbf{a}_{n,k}^\intercal  X^{(n)}\}_{n\in\NN_k}$ satisfies an LDP in $\RR^k$  with GRF $\mathcal{J}_{\gamma^{\otimes k}}^{\sf{qu}}$.

\item Let $U$ be a uniformly distributed random variable on $[0, 1]$, independent of $\{X^{(n)}\}_{n\in\NN}$. If the log mgf $\Lambda$ of \eqref{lambda} is symmetric in its first argument, then the claims (i) and (ii) also hold for the  sequence
  $\{n^{-1/2} \mathbf{a}_{n,k}^\intercal X^{(n)}U^{1/n}\}_{n\in\NN_k}$.
\end{enumerate}
\end{theorem}

 \begin{remark}
   Claim (iii) of Theorem \ref{th-quenched} is motivated by the observation that
if $X^{(n, p)}$ is distributed according to the cone measure on the scaled $\ell_p^n$ sphere $\mathbb{S}_{n,p}$,
   then the random variable $U^{1/n} X^{(n,p)}$ is uniformly distributed on the scaled $\ell_p^n$ ball $\mathbb{D}_{n,p}$
   \cite{SchZin90}.  Since,  
   as noted in Remark \ref{rem-asq}, $X^{(n, p)}$ satisfies Assumption \ref{as-quenched} across a wide range of $p$ (with symmetric
   $\Lambda$), 
   Theorem \ref{th-quenched}(iii) allows an extension of the LDP results in (i) and (ii) of Theorem \ref{th-quenched} 
   from $\ell_p^n$ spheres to $\ell_p^n$ balls, which are of greater interest in convex geometry. 
   \end{remark}

Note that the  rate function $\mathcal{J}_\nu^{\sf{qu}}$ depends only on the limit $\nu$ in \eqref{eq-qwass},  
and is  insensitive to further specifics of the  projection matrix
 sequence $\mathbf{a}$.
 For one-dimensional projections ($k=1$),  Theorem \ref{th-quenched} recovers  both Theorem 2 of \cite{GanKimRam16}, which addresses 
 the case where $X^{(n)}$ has a  product distribution,  
 and Theorem 2.5 and Proposition 5.3 of \cite{GanKimRam17}, which consider the case when 
 $X^{(n)}$ is  uniformly distributed on $\mathbb{D}_{n,p}$  or according to the cone measure $\eta_{n,p}$   (as defined in Remark \ref{rem-quenched}). 
    One setting of multidimensional projections ($k > 1$) considered prior to the above result 
   is the LDP  for  the  projection
   of $X^{(n)}$ onto the first $k$ canonical directions, where $\mathbf{a}_{n,k}$ is  the  matrix of 1s on the diagonal and 0s elsewhere, which 
   does not satisfy \eqref{eq-qwass}. 
    More recent work \cite{KabProTha19} 
       establishes interesting asymptotics (law of large numbers and LDPs) for the shape of multidimensional 
       projections of the uniform distribution on a cube or discrete cube.
 This paper  differs by establishing  almost everywhere quenched LDP results,  first reported in the PhD thesis  \cite{skim-thesis}, 
   for multidimensional projections beyond the particular cases of the canonical projection and
    product measures. 
       Our results provide a potential 
     starting point for obtaining  asymptotic results for shapes and instrinsic volumes
     of projections of non-product measures such as $\ell_p^n$ balls, as well as 
      for ongoing work on sharp quenched large deviation estimates for multi-dimensional projections and their norms, which are 
       relevant for understanding volumetric properties of convex bodies and  their
       intersections.

   Our  second  main result concerns a variational representation of the annealed rate function for the sequence of  random multi-dimensional
   projections $\{\mathbf{A}_{n,k}^\intercal X^{(n)}\}_{n \in\NN_k}$.
   We start by stating an annealed LDP counterpart to Theorem \ref{th-quenched}, specialized to the setting considered in this article.

 \begin{theorem}[Annealed]  
     \label{th-annealed}
     Consider a  sequence $\{X^{(n)}\}_{n\in \NN}$ that satisfies
     Assumption \ref{as-quenched}(i, iv) with associated 
      $\{\xi_j\}_{j \in \NN}$,  $\rho$, and $\newlambda$.
     Then, for any $k \in \NN$, $\{n^{-1/2} \mathbf{A}_{n,k}^\intercal X^{(n)}\}_{n\in\NN_k}$ satisfies an LDP in $\RR^k$ with GRF $\mathcal{J}^{\sf{an}}: \RR^k \ra [0, \infty]$ defined by 
\begin{equation*}
  \mathcal{J}^{\sf{an}}(x) := \inf_{c > 0} \left\{J_X\left(\tfrac{\|x\|_2}{c}\right)  -\tfrac{1}{2} \log(1- c^2) \right\}, \quad x\in \RR^k, 
\end{equation*}
where $J_X$ is given, in terms of the  Legendre transform  $\newlambda^*$  of $\newlambda$, by 
\begin{equation*}
  J_X (x)  := \inf_{(t_1,t_2) \in  \RR_+^2} \left\{ \newlambda^*(t_1,t_2):   x = t_1^{1/2} \rho(t_2)\right\} =
  \inf_{t_2 > 0} \newlambda^* \left(\frac{x^2}{\rho^2(t_2)}, t_2 \right). 
  \end{equation*} 
\end{theorem}
\begin{proof} 
  It follows from  Theorem 2.7 of \cite{KimLiaRam20}  that
  $\{n^{-1/2} \mathbf{A}_{n,k}^\intercal X^{(n)}\}_{n\in\NN_k}$ satisfies an 
  LDP in $\RR^k$ with GRF  $\mathcal{J}^{\sf{an}}$ as defined above 
  whenever Assumption A*  therein  is satisfied with speed $s_n = n$,
  namely, when the sequence of scaled norms $\{\|X^{(n)}\|_2/\sqrt{n}\}_{n\in\NN}$ satisfies an LDP with GRF  $J_X$.  Since
 the domain of  $\newlambda$ contains a neighborhood of the origin due to 
 Assumption \ref{as-quenched}(iv),    Cram\'{e}r's theorem
 implies that  the sequence
 $\{\left( \frac{1}{n} \sum_{i=1}^n \xi_i^2, \frac{1}{n} \sum_{i=1}^n \mathbf{r}(\xi_i) \right)\}_{n \in \NN}$ 
 satisfies an LDP in $\RR^2$ with GRF $\newlambda^*$. Since
 Assumption \ref{as-quenched}(i) implies 
  \begin{equation*}\frac{\|X^{(n)}\|_2}{\sqrt{n}} \eqdist \left( \frac{1}{n}  \sum_{i=1}^n \xi_i^2 \right)^{1/2} \rho \left( \frac{1}{n}\sum_{i=1}^n \mathbf{r}(\xi_i) \right),
  \end{equation*}
  with $\rho$ continuous, 
  the contraction principle
  shows that the sequence of scaled norms $\{\|X^{(n)}\|_2/\sqrt{n}\}_{n\in\NN}$ satisfies an LDP with GRF  $J_X$.
 This  completes the proof.   
\end{proof}

To state the variational representation  for the rate function $\mathcal{J}^{\sf{an}}$,  
we first introduce some notation. 
 For $\nu, \mu \in \cP(\RR)$, define the \emph{relative entropy} of $\nu$ with respect to $\mu$ as
\begin{equation}\label{relent}
  H(\nu | \mu) := \int_{\RR} \log \left( \tfrac{d\nu}{d\mu}\right) d\nu 
\end{equation}
if $\nu \ll \mu$, and $H(\nu | \mu) := +\infty$ otherwise. Let $\gamma$ denote the standard Gaussian measure on $\RR$, and $\gamma^{\otimes k}$  the associated product measure on $\RR^k$. For $\nu\in\cP(\RR^k)$, let $\mathcal{C}:\cP(\RR^k)\ra \RR^{k\times k}$ denote the covariance map,
\begin{equation}\label{cdef}
\mathcal{C}(\nu) :=   \int_{\RR^k}[x\otimes x]\,\nu(dx)   , \quad \nu\in\cP(\RR^k).
\end{equation}
Let $I_k$ denote the $k\times k$ identity matrix, write $A \preceq B$ if $B-A$ is positive semidefinite, and define the rate function
\begin{equation}\label{hk}
  \mathbb{H}_k(\nu) := \left\{\begin{array}{ll}
H(\nu | \gamma^{\otimes k}) + \frac{1}{2} \textnormal{tr}(I_k - \mathcal{C}(\nu)) & \textnormal{ if } \mathcal{C}(\nu) \preceq I_k \\
+\infty & \textnormal{ else} 
\end{array}\right.,  \quad \nu\in\cP(\RR^k).
\end{equation}
  
\begin{theorem}\label{th-var}
  Fix $k\in\NN$, suppose that the sequence $\{X^{(n)}\}_{n\in\NN}$ satisfies Assumption \ref{as-quenched}. Let $\mathcal{J}_\nu^{\sf{qu}}$ and
 $\mathcal{J}^{\sf{an}}$ be defined as in Theorems \ref{th-quenched} and \ref{th-annealed}, respectively. 
  Then, we have the following variational formula: 
  \begin{equation}
    \label{formula-var}
  \mathcal{J}^{\sf{an}}(x) = \inf_{\nu \in \cP(\RR^k)} \left\{ \mathcal{J}_\nu^{\sf{qu}}(x) + \mathbb{H}_k(\nu)  \right\}, \qquad
  x\in \RR^k. 
\end{equation}
\end{theorem}

Note that  $\mathbb{H}_k(\nu) =0$ when $\nu = \gamma^{\otimes k}$,  which implies   $\mathcal{J}^{\sf{an}} \leq \mathcal{J}^{\sf{qu}}_{\gamma^{\otimes k}}$, as would  be expected from  Jensen's inequality given $\mathcal{J}^{\sf{qu}}_{\gamma^{\otimes k}}$ is simply the GRF of the quenched
LDP for $\{\mathbf{A}_{n,k}^\intercal X^{(n)}\}_{n \in \NN_k}$.
More generally, the optimization problem \eqref{formula-var}
can be interpreted as saying that at the large deviation level,
the annealed probability of a rare event is the infimum, over all random ``environments'' (in this case ``projections''),  of the probability
of the rare event conditioned on that environment plus the cost of the choice of the environment. 
While such a  relation is intuitive,   rigorous proofs of such informal statements are typically  non-trivial. 
For example, such  variational representations have been rigorously established only in  a few specific cases, such as LDPs for random
walks in random environments on $\mathbb{Z}$  in \cite{ComGanZei00} and on supercritical Galton-Watson trees in \cite{Aid10}.  
The one-dimensional case ($k=1$) of Theorem \ref{th-var} for $\ell_p^n$ balls recovers  Theorem 2.7 of \cite{GanKimRam17}. 
The proof of the 
 the multi-dimensional case stated in 
Theorem \ref{th-var}),   which is given in Section \ref{sec-pfva}, is more involved and
relies on an auxiliary LDP for the following sequence
of random empirical measures, analogous to those defined in \eqref{emp_L}: 
\begin{equation}\label{eq-lnk}
  \mathsf{L}_{n,k} := \mathsf{L}_{n,k}^{\mathbf{A}}  = \frac{1}{n}\sum_{i=1}^n \delta_{\sqrt{n}\mathbf{A}_{n,k}(i,\cdot)},\quad n\in\NN.
\end{equation}

\begin{theorem}\label{th-stldp}
Fix  $k \in \NN$. Then for all $q \in (0,2)$, the sequence $\{\mathsf{L}_{n,k}\}_{n\in\NN_k}$ satisfies an LDP in $\cP_q(\RR^k)$ with the strictly convex GRF $\mathbb{H}_k$.
\end{theorem}

This theorem, which is established in  Section \ref{sec-auxemp},  generalizes Theorem 6.6 of \cite{BenDemGui01}, which states an LDP for the empirical measure of coordinates drawn uniformly from the sphere $S^{n-1}$, which corresponds to the case $k=1$ in our work.
In contrast to this case, the $k>1$ case necessitates more extensive computations which arise due to the non-commutative matrix setting, where the Bartlett decomposition of Proposition \ref{prop-bartlett} replaces the usual polar decomposition for a random vector from the sphere. Given that large deviation perspectives have informed the analysis of asymptotics for spherical integrals \cite{GuiMad05, OnaMorHal13}, it is possible that a similar approach could inform asymptotics for integrals over the Stiefel manifold, which arise, for instance, as the normalizing constant of the matrix Bingham distribution \cite{Hof09}, or in the study of multi-spiked random covariance matrices \cite{OnaMorHal14}.

\begin{remark}
  The first term in the definition \eqref{hk} of $\mathbb{H}_k$ is the relative entropy with respect to the $k$-dimensional
  standard Gaussian measure, which is (due to Sanov's theorem) the large deviation rate function for the sequence of empirical
  measures of the rows of an $n\times k$ matrix of i.i.d.\ standard Gaussian elements. Hence, Theorem \ref{th-stldp} offers a way of characterizing the distinction between Haar-distributed  matrices  on the Stiefel manifold and Gaussian matrices. Outside of the large deviations literature, a different comparison between such Stiefel and Gaussian matrices can be found in \cite{Tro12}, which analyzes  expectations of sublinear convex functions of random matrices. 
\end{remark}

\begin{remark}
  The second term in the definition \eqref{hk} of $\mathbb{H}_k$ arises from the orthogonality and normalization constraint defining the Stiefel manifold.
  Note that because $\PP(\mathbf{A}_{n,k}^\intercal \mathbf{A}_{n,k} = I_k) = 1$, we have, $\PP\text{-a.s.}$,
\begin{align}\label{traceterm}
0 =  \text{tr}\left(I_k - \mathbf{A}_{n,k}^\intercal \mathbf{A}_{n,k}\right) =\text{tr}(I_k- \mathcal{C}(\mathsf{L}_{n,k})), \quad n\in\NN_k. 
\end{align} 
Nonetheless, the definition of the rate function $\mathbb{H}_k$ includes the trace term $\text{tr}(I_k - \mathcal{C}(\nu))$, and $\mathbb{H}_k(\nu)$ is finite even for $\nu \in \cP(\RR^k)$ such that $\text{tr}(I_k - \mathcal{C}(\nu))\ne 0$, due to the fact that the statement of
an LDP (Definition \ref{def-ldp}) involves infimization of the rate function $\mathbb{H}_k$ not over a set like $\mathcal{V}_k:= \{\nu\in\cP(\RR^k) : I_k = \mathcal{C}(\nu)\}$, but rather over its interior and closure (in the space of probability measures). In particular, the example set $\mathcal{V}_k$ is neither open nor closed with respect to the weak topology.  In fact, it is possible to show from \cite{WanWanWu10} that  $\mathcal{V}_k$ is neither open nor closed with respect to \emph{any} topology for which the sequence
 of empirical measures $\{\mathsf{L}_{n,k}\}_{n\in\NN_k}$ satisfies an LDP. 
\end{remark}

An immediate consequence of Theorem \ref{th-stldp} is the following:

\begin{corollary}\label{slln}
  Fix  $k \in \NN$. Then for all $q \in (0,2)$, the sequence $\{\mathsf{L}_{n,k}\}_{n\in\NN_k}$ satisfies the strong law of large numbers in $\cP_q(\RR^k)$. That is, almost surely, as $n \to \infty$, we have $\cW_q(\mathsf{L}_{n,k}, \gamma_2^{\otimes k}) \to 0$.  
\end{corollary} 
\begin{proof}
By Theorem~\ref{th-stldp}, the rate function $\mathbb{H}_k$ in~\eqref{hk} is strictly convex. Since $\mathbb{H}_k(\gamma^{\otimes k})=0$, $\mathbb{H}_k$ attains its unique minimum over $\cP_q(\RR^k)$ at $\gamma_2^{\otimes k}$. For $\epsilon >0$, due to the LDP for $\{\mathsf{L}_{n,k}\}_{n\in\NN_k}$ and the uniqueness of the minimum of $\mathbb{H}_k$, there exists $\delta >0$ and $N\in\NN_k$ such that for $n>N$, $\PP(\mathcal{W}_q(\mathsf{L}_{n,k},\gamma_2^{\otimes k})>\epsilon )\leq e^{-n\delta}$, which combined with the Borel-Cantelli Lemma yields the almost sure convergence of $\mathsf{L}_{n,k}$.
\end{proof}

\section{The Bartlett decomposition and its consequences}\label{sec-bartlett}

We recall the well known result of Bartlett on the \textsf{QR} decomposition of a matrix with independent standard Gaussian entries, and derive
certain consequences which will be used in the proofs of the main theorem. 
Throughout, let $\tk$ denote the space of $k\times k$ upper triangular matrices, and recall that $\vnk$ denotes the Stiefel manifold of $k$-frames in $\RR^n$.

\begin{definition} 
Fix $k\le n\in\NN$.  Let $\znk\in\RR^{n\times k}$ be an $n\times k$ matrix with i.i.d.\ standard Gaussian elements. Let $\znk=\qnk\rnk$ be the \textnormal{\sffamily QR} decomposition of $\znk$ as the product of the semi-orthogonal matrix $\qnk\in \vnk \subset \RR^{n\times k}$ and the upper triangular matrix $\rnk\in \tk \subset \RR^{k\times k}$. \end{definition}

\begin{proposition}[Bartlett decomposition \cite{Bar33}]\label{prop-bartlett}
The law of $\qnk$ is $\sigma_{n,k}$, the Haar measure on $\vnk$. Moreover, the diagonal entries of $\rnk$ satisfy $\rnk(i,i) \sim \chi_{n-i+1}$, the chi distribution with $n-i+1$ degrees of freedom, for $i=1,\cdots, k$. 
\end{proposition}

\begin{remark}
In fact, the matrices $\qnk$ and $\rnk$ of the Bartlett decomposition are independent, and moreover, the marginal law of the off-diagonal entries of $\rnk$ are also explicitly known; however, we will not need these facts for our analysis. Also,  note that when $k=1$, the Bartlett decomposition corresponds to the classical polar decomposition of the $n$-dimensional Gaussian measure.
\end{remark}

Let $\lkz$ denote the empirical measure of the \emph{rows} of $\znk$,
\begin{equation*} \lkz := \frac{1}{n}\sum_{i=1}^n \delta_{\znk(i,\cdot)} \in \cP(\RR^k).
\end{equation*}
Then, due to Proposition \ref{prop-bartlett}, for $\mathbf{A}_{n,k}$ distributed according to the Haar measure $\sigma_{n,k}$ on $\vnk$, we  have 
\begin{equation}\label{xzr} \mathbf{A}_{n,k} \eqdist \qnk =\znk \rnk^{-1}.
\end{equation}
In the second equality, we use the fact that $\rnk$ is (almost surely) invertible, since it is an upper triangular matrix with diagonal entries that are all (almost surely) positive.  Recalling the definition of $\lk$ from  \eqref{eq-lnk}, and using 
 the representation \eqref{xzr}, we have for any Borel set $B\subset \RR^k$, 
\begin{align}
  \lk(B) &\eqdist \frac{1}{n} \sum_{r=1}^n \delta_{\sqrt{n} \znk(r,\cdot) \rnk^{-1}} (B)
  = \lkz \left(B \tfrac{\rnk}{\sqrt{n}}\right). \label{empalt}
\end{align}
Fortuitously, each element of the matrix $\rnk$ can be computed as a function of the \emph{rows} of the matrix $\znk$, and  can be written as the image of a linear functional of the measure $\lkz$. To be more precise, we recall the Gram-Schmidt process. For a matrix $Z\in \RR^{n\times k}$ with columns $z_1,\dots , z_k \in \RR^n$, let
\begin{align*}
y_1 &:= z_1  & q_1 := \frac{y_1}{\|y_1\|_2}; \hspace{0.894in}\\
y_i &:=  z_i - \sum_{h=1}^{i-1} \langle q_h,z_i\rangle q_h & q_i := \frac{y_i}{\|y_i\|_2}, \quad i=2,\dots, k.
\end{align*}
Then, we have the decomposition $Z=QR$ where $Q=(q_1,\dots, q_k)$ and 
\begin{equation}\label{rmat} R = \begin{pmatrix} \langle q_1, z_1 \rangle  & \langle q_1, z_2\rangle & \langle q_1, z_3\rangle & \cdots & \langle q_1, z_k \rangle \\ 0 & \langle q_2, z_2\rangle & \langle q_2,z_3 \rangle &  \cdots & \langle q_2, z_k \rangle \\ 0 & 0  & \langle q_3,z_3 \rangle &  \cdots & \langle q_3, z_k \rangle \\   \vdots &   & & \ddots & \vdots   \\   0 &  &  & & \langle q_k, z_k\rangle 
 \end{pmatrix}.
\end{equation}
Note that we also have the following relation among the elements of $R$: for $1\le i\le j \le k$,
\begin{equation}\label{rrel}
\langle q_i, z_j \rangle = \frac{\langle y_i, z_j\rangle }{\|y_i\|_2} = \frac{\langle z_i, z_j \rangle - \sum_{h=1}^{i-1} \langle q_h,z_i\rangle  \langle q_h, z_j\rangle  }{\left(\|z_i\|^2 - \sum_{h=1}^{i-1} \langle q_h,z_i\rangle^2  \right)^{1/2}}.
\end{equation}
This expression allows us to clarify the relationship between $\rnk$ and $\lkz$.

\begin{definition}\label{def-symk}
Let $\text{Sym}_k$ be the space of real symmetric $k\times k$ matrices. For $L,M \in \text{Sym}_k$, we write $L \succeq M$ (resp., $L\succ M$) if $L-M$ is positive semi-definite (resp., positive definite).
\end{definition}

We equip  $\text{Sym}_k \subset \RR^{k \times k}$ with  the induced 
  Borel $\sigma$-algebra  when viewing  it as a measurable space,  and
the Frobenius norm when viewed as a Banach space.

\begin{definition}
Define the map $\Gamma:\RR^{k\times k} \ra \tk$ according to the following iterative procedure: for $M\in \RR^{k\times k}$, $j=1,\dots,k$, and $i=1,\dots, j$, 
\begin{equation}\label{gamdef} \Gamma(M)_{ij} := \frac{M_{ij} - \sum_{h=1}^{i-1} \Gamma(M)_{hi}\, \Gamma(M)_{hj}}{\left(M_{ii} - \sum_{h=1}^{i-1} \Gamma(M)_{hi}^2  \right)^{1/2}} \, .
\end{equation}
In the case $i=j=1$, we abide by the convention that empty summations are set to zero, so $\Gamma(M)_{11} :=  M_{11}^{1/2}$.
\end{definition}

\begin{lemma}\label{lem-rho}
We have $\PP$-a.s.,  $n^{-1/2}\rnk = \Gamma ( \mathcal{C}(\lkz))$, where $\mathcal{C}$ is the covariance map of \eqref{cdef}.
\end{lemma}
\begin{proof}
The result follows from \eqref{rmat} and \eqref{rrel} upon noticing that for $1\le i \le j \le k$,
\begin{align*} \mathcal{C}_{ij}(\lkz) &= \frac{1}{n} \sum_{r=1}^n \znk(r,i)\znk(r,j) = \frac{1}{n} \langle \, \znk(\cdot,i) , \znk(\cdot, j)  \, \rangle.
\end{align*}
\end{proof}

\begin{example}
For example, when $k=1$, we have
 $\frac{\mathbf{R}_{n,1}}{\sqrt{n}}= \mathcal{C}_{11}(\mathsf{L}_{n,1})^{1/2}=\frac{ \| \mathbf{Z}_{n,1} \|_2}{\sqrt{n}}$,  
and for $k=2$, we have
\begin{equation*} \frac{\mathbf{R}_{n,2}}{\sqrt{n}} = \begin{pmatrix} 
 \mathcal{C}_{11}(\mathsf{L}_{n,2}) & \frac{ \mathcal{C}_{12}(\mathsf{L}_{n,2}) }{  \mathcal{C}_{11}(\mathsf{L}_{n,2})^{1/2}  }\\\
 0 & \left(\mathcal{C}_{22}(\mathsf{L}_{n,2}) - \frac{\mathcal{C}_{21}(\mathsf{L}_{n,2})^2}{\mathcal{C}_{11}(\mathsf{L}_{n,2})}  \right)^{1/2}
 \end{pmatrix}.
\end{equation*}
\end{example}

\begin{remark}\label{rmk-cholesky}
Note that if $M$ is symmetric positive semi-definite, then $\Gamma(M)$ computes the \emph{Cholesky decomposition} of $M$, so that $\Gamma(M)^\intercal  \Gamma(M) = M$.
\end{remark}


\section{Proof of the empirical measure large deviations}\label{sec-auxemp}

The representation \eqref{empalt} and Lemma \ref{lem-rho} suggest our plan of attack for the proof of Theorem \ref{th-stldp}: first prove a joint LDP for $\{\lkz,  \mathcal{C} (\lkz) \}_{n\in\NN_k}$; then, establish an LDP for $\{\lk\}_{n\in\NN_k}$. Note that we do not attempt to directly establish an LDP for $\{\lk\}_{n\in\NN_k}$ from that for $\{\lkz\}_{n\in\NN_k}$, because the map $\lkz \mapsto \lk$ is not continuous with respect to the weak topology. Nor do we attempt to directly establish an LDP for $\{\lkz, \Gamma(\mathcal{C}(\lkz))\}_{n\in\NN_k}$, because the map $\Gamma \circ \mathcal{C}$ is a nonlinear functional on the space of measures on $\RR^k$. In contrast, our proposed first step is tractable precisely because $\mathcal{C}$ is a linear functional.
We will make use of the following result, which is stated in \cite[Corollary A.2]{KimLiaRam20} as a corollary of \cite[Proposition 6.4]{BenDemGui01}. 

\begin{lemma}[approximate contraction principle]\label{cor-iid}
Let $\Sigma$ be a Polish space and ${\mathcal X}$ be a separable Banach space with topological dual $\mathcal{X}^*$. Let $\{\mathscr{L}_n\}_{n\in \NN}$ be a sequence of $\cP(\Sigma)$-valued random variables such that for each $n\in \NN$, $\mathscr{L}_n$ is the empirical measure of $n$ i.i.d.\  $\Sigma$-valued random variables $\newxi_1,\dots,\newxi_n$ with common distribution $\mu$ (that does not depend on $n$). For any continuous $W:\Sigma \ra \RR$, define 
\begin{equation}
\label{rel-tillam}
\widehat{\Lambda}(W):= \log \mathbb{E}[e^{W(\newxi_1)}].
\end{equation}
Also, let $\mathbf{c}: \Sigma \mapsto {\mathcal X}$ be a continuous map such
that   $0$ lies in the interior $\mathscr{D}^\circ$ of the set 
 \begin{equation}
 \label{rel-calD}
  \mathscr{D} := \left\{\alpha\in \mathcal{X}^* : \widehat{\Lambda}(\langle \alpha, \mathbf{c}(\cdot) \rangle) < \infty
  \right\},  
 \end{equation}
and let $\mathscr{C}_n := \int_{\Sigma} \mathbf{c}(x)\mathscr{L}_n(dx)$. Lastly, define $F:\cX \ra \RR$ as
\begin{equation}\label{eq-f}
F(x) := \sup_{\alpha \in \mathscr{D}_\circ} \langle \alpha, x \rangle, \quad x \in \cX.
\end{equation}
 Then, $\{\mathscr{L}_n, \mathscr{C}_n\}_{n \in \NN}$ satisfies an LDP with the GRF $\mathbb{I}: \cP(\Sigma) \times \mathcal{X} \rightarrow [0,\infty]$ defined by
\begin{equation}\label{eqdef-irf}   \mathbb{I}(\nu, x) :=  \left\{\begin{array}{ll}
 H(\nu|\mu) + F\left( x - \int_\Sigma \mathbf{c} \,d\nu\right) & \textnormal{ if }  H(\nu|\mu) < \infty,\\
+\infty & \textnormal{ else,} 
 \end{array}\right. \quad \nu\in\cP(\Sigma), x\in \mathcal{X}.
\end{equation}
\end{lemma}

\begin{lemma}
For any $k \in \NN$, the sequence $\{\lkz, \mathcal{C}(\lkz)\}_{n\in\NN_k}$ satisfies an LDP in $\cP(\RR^k) \times  \textnormal{Sym}_k$ with GRF $\mathbb{J}_k: \cP(\RR^k) \times \textnormal{Sym}_k \rightarrow [0,\infty]$, defined, for $\nu \in \cP(\RR^k)$ and $M \in \textnormal{Sym}_k$, to be
\begin{equation}\label{jform}
\mathbb{J}_k(\nu,M) := \left\{\begin{array}{ll}
H(\nu |\gamma^{\otimes k}) + \tfrac{1}{2} \textnormal{tr}\left(M - \int_{\RR^k} [z \otimes z] \, \nu(dz) \right) & \textnormal { if } \int_{\RR^k} [z \otimes z] \,\nu(dz) \preceq M,\\
 +\infty & \textnormal{ else.}
\end{array}\right.
\end{equation}
\end{lemma}
\begin{proof} 
We invoke the approximate contraction principle of Lemma \ref{cor-iid} with the following parameters:
$\Sigma = \RR^k$; $\cX = \text{Sym}_k$ and $\cX^* = \text{Sym}_k$;
$\mathbf{c}(z) := [z\otimes z]$ for $z\in \RR^k$;
$\mathscr{L}_n :=   \frac{1}{n} \sum_{j=1}^{n} \delta_{\newxi_j} \mbox{ for }
\newxi_1,\newxi_2,\dots$ i.i.d.\ random vectors with common distribution  $\gamma^{\otimes k};$
and  $\mathscr{C}_n := \int_{\RR^k}\mathbf{c}\, d\mathscr{L}_n$.
Note that 
\begin{equation} \label{observation}
(\lkz,\mathcal{C}(\lkz))   \eqdist (\mathscr{L}_n, \mathscr{C}_n), \qquad n \in \NN_k. 
\end{equation}
With $\widehat{\Lambda}$ as defined in \eqref{rel-tillam}, 
the domain $\mathscr{D}$ specified in \eqref{rel-calD} takes the form
\begin{align*} \mathscr{D} &= \left\{\zeta \in \text{Sym}_k : \log\int_{\RR^k} \exp\left( \langle \zeta, \mathbf{c}(z)\rangle \right) \gamma^{\otimes k}(dz) < \infty\right\}\\    &= \left\{\zeta \in \text{Sym}_k : \log\int_{\RR^k}\frac{1}{(2\pi)^{k/2}} \exp\left( -z^\intercal  (\tfrac{1}{2} I_k - \zeta )z\right) dz  < \infty\right\}\\    &=   \left\{\zeta \in \text{Sym}_k :  \tfrac{1}{2} I_k - \zeta \succ 0\right\}.
\end{align*}
This last expression indicates that $\mathscr{D}$ is a shifted reflection of the positive definite cone, hence open, implying that $\mathscr{D}^\circ = \mathscr{D}$. This expression for the form of $\mathscr{D}$ also makes it clear that $0 \in \mathscr{D}^\circ$. Lastly the value of the supremum in the definition of $F$ takes an explicit form due to the linearity of the trace functional and the fact that the constraint set $\mathscr{D}$ is a cone (due to the positive definiteness constraint): for $\eta \in \text{Sym}_k$, 
\begin{equation*} F(\eta) = \sup_{\zeta \in \text{Sym}_k} \left\{ \text{Tr}(\zeta \eta) : \zeta \prec \tfrac{1}{2} I_k\right\} = \left\{ \begin{array}{cl}
 \frac{1}{2}\text{Tr}(\eta) & \text{ if } \eta \succeq 0,\\
 +\infty & \text{ else.}
 \end{array}\right. 
\end{equation*}
Therefore, \eqref{observation} and Lemma \ref{cor-iid} together imply that
$\{(\lkz,\mathcal{C}(\lkz))\}_{n \in \NN_k}$  satisfies the stated LDP.  
\end{proof}

We now establish a relation between the rate function $\mathbb{J}_k$ and the rate function $\mathbb{H}_k$ of Theorem \ref{th-stldp}.

\begin{lemma} \label{lem-ieq}
For any $k \in \NN$ and  $\nu\in\cP(\RR^k)$, given $\mathbb{H}_k$ of \eqref{hk}, $\mathbb{J}_k$ of \eqref{jform}, and $\Gamma$ of\eqref{gamdef}, we have 
\begin{equation*} \mathbb{H}_k(\nu) = \widetilde{\mathbb{H}}_k(\nu) := \inf_{ M\in  \textnormal{Sym}_k} \mathbb{J}_k(\nu(\,\cdot\times \Gamma(M)^{-1}),M).
\end{equation*}
\end{lemma}
\begin{proof} 
As noted in Remark \ref{rmk-cholesky},  for $M\in \text{Sym}_k$ and $\Gamma$ as in \eqref{gamdef}, we have $M = \Gamma(M)^\intercal  \Gamma(M)$. Given this equality, the constraint 
$M\succeq \int_{\RR^k} [x \otimes x] \, \nu(dx \times \Gamma(M)^{-1})$  
can be rewritten, using the notation $\mathcal{C}$ from \eqref{cdef}, as 
\begin{equation*} I_k \succeq \int_{\RR^k} [x \otimes x] \nu(dx) = \mathcal{C}(\nu).
\end{equation*}
If the preceding constraint is satisfied, then using the form of $\mathbb{J}_k$ in \eqref{jform} in the first equality below, the chain rule for relative entropy in the third equality, and then the form of the Gaussian distribution $\gamma^{\otimes k}$, we obtain 
\begin{align*}
\mathbb{J}_k &(\nu(\,\cdot\times \Gamma(M)^{-1}),M) \\
 &= \frac{1}{2}\sum_{i=1}^k\left(M_{ii} - \int_{\RR^k}x_i^2 \, \nu(dx \times \Gamma(M)^{-1}) \right) +  H(\nu(\,\cdot\times \Gamma(M)^{-1}) | \gamma^{\otimes k} ) \\
 &= \frac{1}{2}\sum_{i=1}^kM_{ii} - \frac{1}{2}\int_{\RR^k}y^\intercal \,\Gamma(M)\Gamma(M)^\intercal y\,\nu(dy)  +  H(\nu | \gamma^{\otimes k}(\,\cdot \times \Gamma(M)) )  \\
  &= \frac{1}{2}\sum_{i=1}^kM_{ii} - \frac{1}{2}\int_{\RR^k}y^\intercal \,\Gamma(M)\Gamma(M)^\intercal y\,\nu(dy)  -  \int_{\RR^k} \log(\tfrac{d\gamma^{\otimes k}(\,\cdot \times \Gamma(M))}{d\gamma^{\otimes k}})\, d\nu  + H(\nu | \gamma^{\otimes k})\\ 
  &= \frac{1}{2}\sum_{i=1}^kM_{ii} - \frac{1}{2}\int_{\RR^k}y^\intercal \,\Gamma(M)\Gamma(M)^\intercal y\,\nu(dy)  + \frac{1}{2}\log\det(\Gamma(M)^{-\intercal}\Gamma(M)^{-1})\\ & \quad \quad \quad+ \frac{1}{2} \int_{\RR^k} y^\intercal \, (\Gamma(M)\Gamma(M)^\intercal  - I_k)y \,\nu(dy)  + H(\nu | \gamma^{\otimes k})
 \end{align*}
 Due to the upper triangular structure of $\Gamma(M)$, we have 
 \begin{equation*} \frac{1}{2}\log\det(\Gamma(M)^{-\intercal}\Gamma(M)^{-1}) = -\log \det (\Gamma(M)) = -\sum_{i=1}^k \log \Gamma(M)_{ii}.
\end{equation*}
 Also, note that $\text{Tr}(I_k - \mathcal{C}(\nu)) = k - \int_{\RR^k} y^\intercal  y\, \nu(dy)$.
Hence, invoking the definition of $\mathbb{H}_k$ in  \eqref{hk},  
  we have
\begin{align*}
\mathbb{J}_k (\nu(\,\cdot\times \Gamma(M)^{-1}),M)   &=  \frac{1}{2}\sum_{i=1}^k(M_{ii} -1)   -\sum_{i=1}^k \log\Gamma(M)_{ii}  + \mathbb{H}_k(\nu).
\end{align*} 
Taking the infimum of the expression above over $M\in \textnormal{Sym}_k$, we see that 
\begin{align}
  \label{temp-H}
\widetilde{\mathbb{H}}_k(\nu) &= \inf_{M\in\textnormal{Sym}_k}\left\{ \sum_{i=1}^k\left( \frac{M_{ii}-1}{2}   - \log \Gamma(M)_{ii} \right)\right\}  + \mathbb{H}_k(\nu).
\end{align}
Note that by the definition of $\Gamma$ in \eqref{gamdef},
\begin{equation*}
\Gamma(M)_{ii} = \left(M_{ii} - \sum_{h=1}^{i-1} \Gamma(M)_{hi}^2 \right)^{1/2},
\end{equation*}
and by the definition of the Gram-Schmidt process, we have $M_{ii} \ge \sum_{h=1}^{i-1} \Gamma(M)_{hi}^2 \ge 0$. Thus, for all $i=1,\dots, k$, for any fixed $M_{ii}$, the maximum value of $\Gamma(M)_{ii}$ is attained when $\Gamma(M)_{hi}=0$ for $h=1,\dots, i-1$. Therefore,
once again using $M = \Gamma(M)^\intercal  \Gamma(M)$, we obtain 
\begin{align*}
 \inf_{M\in\textnormal{Sym}_k}\left\{ \sum_{i=1}^k\left( \frac{M_{ii}-1}{2}   - \log \Gamma(M)_{ii} \right)\right\}  &= \inf_{M_{ii}\ge 0, i=1,\dots, k} \left\{ \frac{1}{2} \sum_{i=1}^k \left(M_{ii}-1 - \log M_{ii}\right) \right\} \\  &=   \frac{1}{2} \sum_{i=1}^k \inf_{M_{ii}\ge 0 } \left\{M_{ii}-1 - \log M_{ii} \right\}, 
\end{align*}
which is clearly equal to zero.  
Together with \eqref{temp-H}, this shows that $\mathbb{H}_k = \widetilde{\mathbb{H}}_k$.
\end{proof}

\begin{lemma}\label{lem-compcat}
Fix $k \in \NN$ and consider the following set of probability measures,
\begin{equation}\label{k2}
  \cK := \left\{ \nu \in \cP(\RR^k) : \int_{\RR^k} \|x\|^2 \nu(dx) \le k\right\}.
\end{equation}
For any $q\in (0, 2)$, the set $\cK\subset \cP_2(\RR^d)$ is compact with respect to the $q$-Wasserstein topology. In addition, $\cK$ is convex and non-empty.
\end{lemma}

\begin{proof}
  The proof is an elementary modification of the proof of the $k=1$ case given in \cite[Lemma 3.14]{KimRam18}.
\end{proof}

\begin{proof}[Proof of Theorem \ref{th-stldp}]
  Let $\Gamma$ be as defined in \eqref{gamdef}. Due to \eqref{empalt} and Lemma \ref{lem-rho}, we have
\begin{equation*} \lk \eqdist \lkz(\,\cdot \times \tfrac{\rnk}{\sqrt{n}})  = \lkz\left(\,\cdot \times \Gamma\left(\mathcal{C}(\lkz)\right)\right).
\end{equation*}
The image of $\mathcal{C}$ is positive semi-definite matrices, so as noted in Remark \ref{rmk-cholesky}, the map $\Gamma$ maps a matrix to its Cholesky decomposition, hence  $M \mapsto \Gamma(M)$ is continuous. By Slutsky's theorem, the map 
\begin{equation*} \cP(\RR^k)\times \textnormal{Sym}_k \ni (\mu,M) \mapsto \mu(\,\cdot \,\times \Gamma(M)) \in \cP(\RR^k),
\end{equation*}
is also continuous. An application of the contraction principle
to the map above yields an LDP for the sequence $\{\lk\}_{n\in\NN_k}$ in $\cP(\RR^k)$ (i.e., with respect to the weak topology), with GRF  
\begin{align*}
 \inf_{\nu\in\cP(\RR^k), M\in  \textnormal{Sym}_k} \{ \mathbb{J}_k(\mu,M) : \nu = \mu(\,\cdot\times \Gamma(M))\} = \inf_{ M\in  \textnormal{Sym}_k} \mathbb{J}_k(\nu(\,\cdot\times \Gamma(M)^{-1}),M) = \mathbb{H}_k(\nu),
\end{align*}
where the last equality is due to Lemma \ref{lem-ieq}.

Fix $q \in (0,2)$. In order to establish the LDP for $\{\lk\}_{n\in\NN_k}$ in $\cP_q(\RR^k)$ (i.e., with respect to the stronger $q$-Wasserstein topology), by Corollary 4.2.6 of \cite{DemZeiBook}, 
it suffices to show exponential tightness of $\{\lk\}_{n\in\NN_k}$ in the $q$-Wasserstein topology.  Let $\cK$ be the set defined in \eqref{k2}, which is compact (with respect to the $q$-Wasserstein topology)  due to Lemma \ref{lem-compcat}. Note that $\int_{\RR^k} |x|^2 \lk(dx) = k$ a.s., so $\PP(\lk \in \cK_{2,k}^c) = 0$, and
hence, $\log \PP(\lk \in \cK_{2,k}^c) = -\infty$ for all $n\in\NN_k$, trivially implying the exponential tightness of $\{\lk\}_{n\in\NN_k}$.

Lastly, the strict convexity of $\mathbb{H}_k$ follows from the strict convexity of the relative entropy $H(\cdot |\gamma^{\otimes k})$ and the linearity of the covariance map $\mathcal{C}$.
\end{proof}


\section{Proof of the quenched LDP}
\label{sec-quenched}

In this section, we state the proof of Theorem \ref{th-quenched}. As a precursor, we state two lemmas that will assist with part (iii) of the theorem. 

\begin{lemma}\label{lem-monotone}
Fix $m\in\NN$, and let $\mathscr{F}$ be a set of functions from $\RR^m$ to $\RR$ such that  every $f\in \mathscr{F}$ is symmetric about 0 and convex. Then, defining $g: \RR^m \ra \RR$ as
\begin{equation*}
g(x) := \inf_{f\in\mathscr{F}} f(x), \quad x\in \RR^m,
\end{equation*}
the function $g$ is monotone with respect to scaling in the sense that for all $x\in \RR^m$, the mapping
\begin{equation}\label{mono}
\RR_+ \ni c \mapsto g(cx) \in \RR 	
\end{equation}
is non-decreasing.
\end{lemma}
\begin{proof}
Fix $x \in \RR^m$ and $c_1 < c_2 \in \RR_+$. For any $f \in \mathscr{F}$, the symmetry about 0 and convexity of $f$ implies that $f$ has a global minimum at 0, hence
\begin{align*}
f(c_1 x) &= f(\tfrac{c_1}{c_2}\times  c_2 x + \tfrac{c_2 - c_1}{c_2}\times 0 ) \\
  &\le  \tfrac{c_1}{c_2} f(c_2x) + \tfrac{c_2 - c_1}{c_2} f(0) \\
  &= f(c_2x) + \tfrac{c_2 - c_1}{c_2} (f(0) - f(c_2x))\\
  &\le f(c_2x),
\end{align*}
where the first inequality follows from convexity, and the second inequality is due to the global minimum at 0. Taking the infimum over all $f\in\mathscr{F}$ on both sides, we find that $g(c_1x) \le g(c_2 x)$, completing the proof.
\end{proof}

\begin{lemma}\label{lem-u1n}
Fix $m\in\NN$, and let $\mathsf{x} = \{\mathsf{x}_n\}_{n\in\NN}$ denote a sequence of $\RR^m$-valued random variables that satisfies an LDP with GRF $I_\mathsf{x}$. Let $U$ be a uniformly distributed random variable on $[0,1]$ independent of $\{\mathsf{x}_n\}_{n\in\NN}$. If for all $x\in \mathbb{R}^m$, the mapping $\RR_+ \ni c\mapsto I_\mathsf{x}(cx) \in [0, \infty]$ is  non-decreasing, then the scaled sequence $\{U^{1/n}\mathsf{x}_n\}_{n\in\NN}$ satisfies an LDP with GRF $I_\mathsf{x}$.
\end{lemma}
\begin{proof}
Due to \cite[Lemma 3.3]{GanKimRam17}, the sequence $\{U^{1/n}\}_{n\in \NN}$ satisfies an LDP with the good rate function
\begin{equation*}
  I_U(u) := \left\{ \begin{array}{ll}
 - \log u & u\in(0,1] ; \\
   +\infty & \textnormal{ else. }
 \end{array}\right.
\end{equation*}
By independence, the sequence $\{U^{1/n},\,\,\mathsf{x}_n\}_{n\in \NN}$ satisfies a joint LDP with the GRF $I_{U, \mathsf{x} }: \RR\times\RR^m$ defined as $I_{U, \mathsf{x} }(u,w) :=I_U(u) + I_\mathsf{x}(x)$. By the contraction principle, the scaled sequence $\{U^{1/n}\mathsf{x}_n\}_{n\in \NN}$ satisfies an LDP with the rate function $I$, where for $x\in \RR^m$,
\begin{align*}
  I(x) &:= \inf_{u\in \RR, \, y\in \RR^m}\{ I_U(u) + I_{\mathsf{x}}(y) : uy = x\}
  = \inf_{u\in(0,1]}\{ -\log u + I_{\mathsf{x} }(\tfrac{x}{u}) \}.
\end{align*}
The mapping $u\mapsto 1/u$ is monotonically decreasing, which when combined with the assumption on $I_\mathsf{x}$ implies that $u\mapsto I_\mathsf{x}(x/u)$ is monotonically decreasing. Since $u\mapsto -\log u$ is also monotonically decreasing, the infimum above is attained at $u = 1$, hence $I(x) = I_\mathsf{x}(x)$ for all $x\in \RR^m$.
\end{proof}

\begin{proof}[Proof of Theorem \ref{th-quenched}]
  Suppose Assumption \ref{as-quenched} holds for some $\{\xi_j\}_{j \in \NN}$, $\mathbf{r}$, $\rho$, $q_\star  > 0$, and $T \le \infty$, all as defined in the statement of the assumption. We first prove an LDP for the following $\RR^{k+1}$-valued sequence: 
\begin{equation}\label{rnp2}
  \Rna := \left( n^{-1/2}\,\mathbf{a}_{n,k}^\intercal  \xi^{(n)} , \,\,\frac{1}{n}\sum_{i=1}^n \mathbf{r}(\xi_i) \right), \qquad n \in \NN. 
\end{equation}
In  terms of the log mgf $\Lambda$ of $(\xi_1, \mathbf{r}(\xi_1))$,  the scaled log mgf of  $\Rna$ takes the form:
for $t_1 \in \RR^k$ and $t_2 \in \RR$, 
\begin{align*}
\frac{1}{n}\log\EE\left[\exp(n\,\langle t, \Rna\rangle)\right] &=  \frac{1}{n} \log \EE\left[ \exp\left(\sum_{i=1}^n   \left( \sqrt{n}\, \xi_i \langle t_1, \mathbf{a}_{n,k}(i,\cdot)\rangle + t_2 \,\mathbf{r}(\xi_i)  \right) \right) \right]\\
&=  \frac{1}{n} \log \prod_{i=1}^n \EE\left[ \exp\left( \sqrt{n} \, \xi_i  \langle t_1, \mathbf{a}_{n,k}(i,\cdot)\rangle + t_2 \, \mathbf{r}(\xi_i) \right)  \right]\\
&= \frac{1}{n} \sum_{i=1}^n \Lambda\left( \, \langle t_1, \sqrt{n} \mathbf{a}_{n,k}(i,\cdot)\rangle, t_2 \right) \\
&= \Psi_{\mathsf{L}_{n,k}^{\mathbf{a}}} (t_1, t_2), 
\end{align*}
with $\Psi_{\cdot}$ equal to the integrated log mgf functional  defined in \eqref{psi}.

Fix $t_1\in \RR^k$. For $t_2 \ge T$, both sides equal $+\infty$ due to Remark \ref{rem-quenched}. For $t_2 < T$, due to the $q_\star$-Wasserstein continuity pointed out in Remark \ref{rem-quenched}, together with the $q_\star$-Wasserstein convergence of \eqref{eq-qwass}, we take the limit as $n \rightarrow \infty$ to find
   \begin{align*}
    \lim_{n\ra\infty}\frac{1}{n}\log\EE\left[\exp(n\,\langle t, \Rna\rangle)\right] &=\Psi_\nu (t_1, t_2).  
  \end{align*}
Due to the lower semicontinuity and essential smoothness of $\Psi_{\nu}$, which follow from Assumption \ref{as-quenched}(iii), the G\"artner-Ellis theorem (see, e.g., \cite[Theorem 2.3.6]{DemZeiBook}) yields the LDP for the sequence $\{\Rna\}_{n\in\NN}$ in $\RR^{k+1}$ with the GRF $\Psi_{\nu}^*$ from \eqref{psistar}.   

Due to the representation of $X^{(n)}$ given by Assumption \ref{as-quenched}(i), we have
\begin{equation*}
  n^{-1/2}\mathbf{A}_{n,k}^\intercal X^{(n)} \eqdist  \rho\left((\Rna)_2\right)(\Rna)_1,  
\end{equation*}
where $\rho$ is continuous. The LDP for $\{\Rna\}_{n\in\NN}$ and the contraction principle applied to the continuous 
mapping $\RR^{k+1} \ni (\tau_1, \tau_2) \mapsto  \rho(\tau_2) \tau_1  \in \RR^k$ yield an LDP for $\{n^{-1/2} \mathbf{a}_{n,k}^\intercal X^{(n)}\}_{n\in\NN_k}$ in $\RR^k$ with good rate function  $\bar{\mathcal{J}}_\nu^{\sf{qu}}$ defined to be
\begin{equation*}
  \bar{\mathcal{J}}_\nu^{\sf{qu}}(x) := \inf_{\tau_1\in\RR^k,\,\tau_2\in \RR_+} \left\{ \Psi_\nu^*(\tau_1,\tau_2) : \tau_1 \rho(\tau_2) = x \right\} ,\quad x\in \RR^k. 	
\end{equation*}
The equivalence of $\bar{\mathcal{J}}_\nu^{\sf{qu}}$ to the rate function $\mathcal{J}_\nu^{\sf{qu}}$ in \eqref{def-jqu} follows from using the constraint $\tau_1 \rho(\tau_2) = x$ to substitute for the first argument in $\Psi_\nu^*$.  This proves part (i) of the theorem. 

In turn, the LDP from part (i) implies part (ii) of the theorem since by Corollary \ref{slln}, almost surely, 
${\mathcal W}_{q_\star}(\lk^\mathbf{A}, \gamma^{\otimes k}) =  {\mathcal W}_{q_\star}(\lk, \gamma^{\otimes k})\rightarrow 0$ as $n \rightarrow \infty$. 

We turn to the final claim (iii). Given the assumption on symmetry of $\Lambda$, it is apparent from the definition \eqref{psi} that $\Psi_\nu$ is symmetric in its first argument, and then from the definition of the Legendre transform \eqref{psistar} that $\Psi_\nu^*$ is also symmetric in its first argument. Applying Lemma \ref{lem-monotone} with dimension $m = k$, the set of symmetric convex functions $\mathscr{F} = \{  \RR^k \ni x\mapsto \Psi_\nu^*(\frac{x}{\rho(\tau)}, \tau) \in \RR \}_{\tau \in \RR_+}$, and $g = \mathcal{J}_\nu^{\sf{qu}}$, we find that the mapping $\RR_+ \ni c \mapsto \mathcal{J}_\nu^{\sf{qu}}(c x) \in [0, \infty]$ is non-decreasing. An application of Lemma \ref{lem-u1n}  with $\mathsf{x}_n = n^{-1/2} \mathbf{a}_{n,k}^\intercal X^{(n)}$ for $n\in \NN$ and $I_\mathsf{x} = \mathcal{J}_\nu^{\sf{qu}}$ completes the proof.
\end{proof}

\section{Proof of the variational formula}\label{sec-pfva} 

In this section, we prove Theorem \ref{th-var}, primarily through an application of Theorem \ref{th-stldp} and Sion's minimax theorem \cite{Sio58}. 
We start with preliminary results in Lemmas \ref{lem-bound}, \ref{lem-lmgrep}, and \ref{lem-phistarvar}.  Throughout,
recall the definition of $\mathbb{H}_k$  from \eqref{hk}. 

\begin{lemma}\label{lem-bound}
Suppose Assumption \ref{as-quenched} holds, with associated $T$ and $\Psi_\nu$, $\nu \in \cP(\RR^k)$, and recall the empirical measure $\lk$ from \eqref{eq-lnk}.
For $t_1 \in \RR^k$, $t_2 < T$, and $0 < \newkappa < \infty$,
  the following   condition holds: 
\begin{equation}\label{varbd}
 \limsup_{n\ra\infty} \frac{1}{n} \log \EE\left[ e^{\newkappa n \Psi_{\mathsf{L}_{n,k}}(t_1,t_2)}\right] < \infty\, .
\end{equation}
	
\end{lemma}
\begin{proof}

  Let $\Theta^{(n)} := (\Theta_1^{(n)}, \dots, \Theta_n^{(n)})$ denote a random vector distributed uniformly on the
  Euclidean
  sphere in $\RR^n$ of radius 1. For $t_1 \in \RR^k$, the  random vector $\mathbf{A}_{n,k} t_1$ lies on the
   Euclidean
  sphere in $\RR^n$
  of radius $\|t_1\|_2$ and has a law invariant to orthogonal transformation (due to the law of $\mathbf{A}_{n,k}$); hence, $\mathbf{A}_{n,k} t_1 \eqdist \|t_1\|_2 \Theta^{(n)}$. Fix $t_1\in \RR^k$ and $t_2 < T$,  let $C_{t_2}$ and $q_\star$ be as in  Assumption \ref{as-quenched}(ii), and for $i\in \NN$, define $g_i: \RR_+ \ra \RR_+$ as
\begin{equation*}
  g_i(x) = \exp \left(\newkappa C_{t_2} [1 + (\sqrt{n}\|t_1\|_2 x)^{q_\star} ]\right), \quad x\in \RR_+.
 \end{equation*}
When combined, the relation $\mathbf{A}_{n,k} t_1 \eqdist \|t_1\|_2 \Theta^{(n)}$, the bound of Assumption \ref{as-quenched}(ii) and 
the sub-independence
of $(|\Theta_1^{(n)}|, \dots, |\Theta_n^{(n)}|)$ as given by \cite[Theorem 2.11(2)]{BarGamLozRou10} with $p=2$ therein, yield  
\begin{equation}\label{intermediate}
\EE\left[ e^{\newkappa n \Psi_{\mathsf{L}_{n,k}}(t_1,t_2)}\right] = \EE \left[\prod_{i=1}^n \exp \left( \newkappa \Lambda (\sqrt{n}\|t_1\|_2 \Theta_i^{(n)}, t_2) \right) \right] \le \EE\left[\prod_{i=1}^n g_i(|\Theta_i^{(n)}|) \right]   \le   \prod_{i=1}^n \EE\left[ g_i(|\Theta_i^{(n)}|)\right].
\end{equation}
Now, let $Z^{(n)} := (Z_1, \dots, Z_n)$,  and note that for each $i = 1, \ldots, n$, $\sqrt{n}\Theta_i^{(n)} \eqdist \sqrt{n}Z_1 / \|Z^{(n)}\|_2 \xrightarrow{\textnormal{a.s.}} Z_1$. Therefore, taking the limit superior, as $n \to \infty$, in  \eqref{intermediate} and applying the reverse Fatou lemma, we obtain 
\begin{align*}
\limsup_{n\ra \infty} \frac{1}{n}\log \EE\left[ e^{\newkappa n \Psi_{\mathsf{L}_{n,k}}(t_1,t_2)}\right]  &\le  \limsup_{n\ra \infty} \log \EE\left[\exp\left(\delta C_{t_2}[1 + (\|t_1\|_2 \tfrac{\sqrt{n}|Z_1|}{\|Z^{(n)}\|_2})^{q_\star}]\right)\right] \\
  &\le \newkappa C_{t_2} +   \log \EE\left[\exp(\delta C_{t_2}\|t_1\|_2^{q_\star} \, |Z_1|^{q_\star}) \right]. 
\end{align*}
Since the last term on the right-hand side is finite for all $t_1 \in \RR^k$ because $q_\star < 2$, \eqref{varbd} follows. 
\end{proof}

\begin{lemma} \label{lem-lmgrep}
  Suppose Assumption \ref{as-quenched} holds, with associated 
  $\{\xi_j\}_{j\in \NN}$, $\mathbf{r}$, $T$, and $\Psi_\nu$, $\nu \in \cP(\RR^k)$, 
  and let $\mathbf{A}_{n,k}$ be drawn from the Haar measure $\sigma_{n,k}$ on $\vnk$, independently of $\{\xi_j\}_{j\in \NN}$. For $n\in \NN$, define
 \begin{equation}\label{phin} \Phi_n(t_1,t_2) :=
    \frac{1}{n} \log \EE \left[\exp\left(\sqrt{n}\,  \xi^{(n)} \mathbf{A}_{n,k} t_1 + t_2 \sum_{i=1}^n \mathbf{r}(\xi_i)   \right)  \right], \quad t_1 \in \RR^k, \, t_2 \in \RR,
 \end{equation}
 where $\xi^{(n)} := (\xi_1, \xi_2, \ldots, \xi_n)$. 
 Then,  for $t_1\in \RR^k$ and $t_2 \in \RR$,  
  \begin{equation*}
 \lim_{n \rightarrow \infty} \Phi_n(t_1,t_2) = \Phi(t_1,t_2), 
\end{equation*}
where, with $\cK$ equal to the set defined in \eqref{k2}, 
 \begin{equation}\label{eq-phivar} \Phi(t_1,t_2) := \sup_{\nu \in \cP(\RR^k)} \left\{ \Psi_\nu(t_1,t_2) - \mathbb{H}_k(\nu)\right\}
   = \sup_{\nu \in \cK} \left\{ \Psi_\nu(t_1,t_2) - \mathbb{H}_k(\nu)\right\}. 
 \end{equation}
  \end{lemma}

\begin{proof}
Due to the independence of $\xi_1,\xi_2,\dots$, and their independence from $\mathbf{A}_{n, k}$,  we can write, for $n \in \NN$, 
\begin{align*}
 \Phi_n(t_1,t_2) &=  \frac{1}{n} \log \EE \left[\prod_{i=1}^n \EE\left[\exp\left(\sqrt{n}\,  \xi_i \left(\mathbf{A}_{n,k} t_1\right)_i + t_2  \mathbf{r}(\xi_i)   \right) \big | \mathbf{A}_{n,k} \right] \right]\\
  &=
  \frac{1}{n} \log \EE \left[\exp\left(\sum_{i=1}^n\Lambda(\sqrt{n}\,\langle \mathbf{A}_{n,k}(i,\cdot) , t_1 \rangle  ,\, t_2 \right) \right]\\
  &=
  \frac{1}{n} \log \EE \left[\exp\left(n \,\Psi_{\mathsf{L}_{n,k}}(t_1,t_2) \right ) \right],
\end{align*}
for $t_1\in \RR^k$ and $t_2 \in \RR$, 
where $\lk$ is as in \eqref{eq-lnk}. Now, let $T \le \infty$ and $q_\star \in (0,2)$ be as specified in Assumption \ref{as-quenched}(ii).
For $t_2 \geq T$, by Remark \ref{rem-quenched}, both 
$\Phi_n(\cdot,  t_2)$ and $\Phi(\cdot,t_2)$ are identically equal to infinity, and so the limit holds trivially.
Now, suppose $t_2 < T$. Recall from Theorem \ref{th-stldp} that the sequence $\{\mathsf{L}_{n,k}\}_{n\in\NN_k}$ satisfies an LDP in $\cP_{q}(\RR^k)$ for all $q \in (0, 2)$, with the GRF $\mathbb{H}_k$. Due to the bound in Assumption \ref{as-quenched}(ii), the map $\cP(\RR) \ni \nu\mapsto \Psi_\nu(t_1,t_2)\in \RR$ is continuous with respect to the $q_\star$-Wasserstein topology, and we know $q_\star \in (0,2)$ due to Assumption \ref{as-quenched}(v). By Varadhan's lemma  \cite[Theorem 4.3.1]{DemZeiBook} which is applicable due to Lemma \ref{lem-bound}, it follows that the limit of $\Phi_n(t_1, t_2)$ is given by $\Phi(t_1, t_2)$ defined in  \eqref{eq-phivar}.

To complete the proof of the lemma, it only remains to
establish the last equality in \eqref{eq-phivar}, but this is an immediate consequence of the fact that $\mathbb{H}_k (\nu) = \infty$ 
for $\nu \notin \cK$. 
	
\end{proof}

\begin{lemma}\label{lem-phistarvar}
  Suppose Assumption \ref{as-quenched} holds, and for each $\nu \in \cP(\RR^k)$,   
  let  $\Psi_\nu$ be as defined  in \eqref{psi},  let $\Psi_\nu^*$ denote its Legendre transform, as specified in
  \eqref{psistar}, and let $\cK \subset \cP(\RR^k)$ be the set defined in \eqref{k2}. 
  Then the Legendre transform $\Phi^*$ of the function $\Phi$ defined in \eqref{eq-phivar} satisfies,
for $\tau_1\in\RR^k$, and $\tau_2\in \RR$, 
\begin{equation}\label{phistarvar} {\Phi}^*(\tau_1, \tau_2) 
= \inf_{\nu\in \cK} \left\{\Psi_\nu^*(\tau_1,\tau_2) + \mathbb{H}_k(\nu) \right\} = 
\inf_{\nu\in \cP(\RR^k)} \left\{\Psi_\nu^*(\tau_1,\tau_2) + \mathbb{H}_k(\nu) \right\}.
\end{equation}
\end{lemma} 
\begin{proof}
  First, note that the second equality in \eqref{phistarvar} holds because $\mathbb{H}_k(\nu) = \infty$ for $\nu\not\in \cK$.
  Next, fix the following:  
\begin{itemize} 
\item let $\Lambda$, $T$ be as in  Assumption \ref{as-quenched}(ii), and define $\cD_T :=\RR^k \times (-\infty, T)$; 
\item let $q_\star \in (0, 2)$ be as in Assumption \ref{as-quenched}(ii), and let $\cM_{q_\star}(\RR^k)$	denote the space of finite signed measures (not necessarily probability measures) on $\RR^k$, equipped with the $q_\star$-Wasserstein topology.  
\end{itemize}

Fix $\tau = (\tau_1, \tau_2) \in \RR^{k}\times \RR$. Then by the definition  \eqref{psistar} of $\Psi_\nu^*$,
\begin{align}
  \Psi_\nu^*(\tau_1,\tau_2)
  &=
  \sup_{(t_1,t_2)\in\RR^{k}\times \RR} \{ \langle \tau_1,t_1\rangle + \tau_2 t_2  - \Psi_\nu(t_1,t_2) \} \notag\\
   &=  \sup_{(t_1,t_2) \in \cD_T} \{ \langle \tau_1, t_1\rangle + \tau_2 t_2 - \Psi_\nu(t_1,t_2) \} \label{legendrepsi},
\end{align}
where the  second equality holds because, by Remark \ref{rem-quenched},
$\Psi_\nu(t_1,t_2) = \infty$ if $t_2 \geq T$. 
Thus, the right-hand side of \eqref{phistarvar} is equal to 
$\inf_{\nu \in \cK} \sup_{t=(t_1,t_2) \in \cD_T} \inf_{\nu \in \cK} F_{\tau} (\nu,t),$ 
where
\begin{equation*} F_{\tau}(\nu, t) := \langle \tau_1, t_1\rangle  + \tau_2 t_2 - \Psi_\nu(t_1,t_2) + \mathbb{H}_k(\nu), \quad \nu \in \cP(\RR^k), \,t=(t_1,t_2) \in \RR^{k} \times \RR. 
\end{equation*}
On the other hand, by the definition of $\Phi^*$ and the
representation \eqref{eq-phivar} for $\Phi$, 
\begin{eqnarray*}
  \Phi^*(\tau_1, \tau_2) & = &  \sup_{(t_1,t_2) \in \RR^{k+1}} \{ \langle \tau_1, t_1 \rangle   + \tau_2 t_2 - \Phi(t_1,t_2)\}
  \\ & = &
  \sup_{t=(t_1,t_2) \in \RR^{k+1}} \inf_{\nu \in \cK} F_{\tau} (\nu,t),\\
  & = & \sup_{t=(t_1,t_2) \in \cD_T} \inf_{\nu \in \cK}  F_{\tau} (\nu,t), 
\end{eqnarray*}
where the last equality uses the fact that 
for $t_2 > T$, $\Psi_\nu(t_1,t_2) = \infty$ and hence, $F_{\tau}(\nu,t) = -\infty$ (see  Remark \ref{rem-quenched}). 
Thus, to prove the
first equality in \eqref{phistarvar}, it suffices to show that for all 
$(\tau_1,\tau_2)  \in \RR^{k}\times \RR$,  
\begin{equation}\label{minmax} \inf_{\nu \in \cK} \sup_{(t_1,t_2) \in\cD_T}
 F_\tau(\nu, (t_1,t_2)) = \sup_{(t_1,t_2) \in \cD_T} \inf_{\nu \in \cK}  F_\tau(\nu, (t_1,t_2)). 
\end{equation}

To justify the exchange of infimum and supremum in \eqref{minmax}, we verify the conditions of the minimax theorem \cite[Corollary 3.3]{Sio58}. That is, for $(\tau_1,\tau_2)\in \RR^{k}\times \RR$, we note that
\begin{itemize} 
\item the set $\cD_T = \RR^k \times (-\infty,T)$ is a convex subset of the topological vector space $\RR^{k+1}$; 
\item due to  Lemma \ref{lem-compcat}, $\cK$ is a convex compact subset of the topological vector space $\cM_{q_\star}(\RR^k)$; 
\item for $t = (t_1,t_2) \in \cD_T$: the lower semicontinuity of $F_\tau(\cdot, t)$ follows from the lower semicontinuity of $\nu \to \Psi_\nu(t)$ due to Assumption \ref{as-quenched}(iii) and of $\mathbb{H}_k$ (as it is a GRF); the convexity of $F_\tau(\cdot, t)$ follows from the linearity of $\nu\mapsto \Psi_\nu(t)$ and the convexity of $\mathbb{H}_k$, which was established in Theorem \ref{th-stldp}; 
\item for $\nu \in \cK$: the lower semicontinuity of $t \to \Psi_\nu(t)$  on $\cD_T$ follows from Assumption \ref{as-quenched}(iii); the convexity of $\Psi_\nu$ on $\cD_T$ follows from linearity of expectation, the definition \eqref{psi}, and the fact that $\Lambda$ is convex since it is a log mgf;
\item
  since $t\mapsto \langle \tau_1, t_1\rangle + \tau_2t_2$ is continuous and linear, it follows that $F_\tau(\nu, \cdot)$ is upper semicontinuous and concave on $\cD_T$.
\end{itemize} 
Due to the conditions verified above, the minimax theorem can be applied to conclude that \eqref{minmax}, and hence, the desired first equality in \eqref{phistarvar}, holds.  This completes the proof of the lemma.
\end{proof} 

\begin{proof}[Proof of Theorem \ref{th-var}]
  Let $\{\xi_j\}_{j \in \NN}$ and $\mathbf{r}$  be as in Assumption \ref{as-quenched}, let $\xi^{(n)}:= (\xi_1, \ldots, \xi_n)$ and 
define the sequence $\{R^{(n)}\}_{n\in\NN}$ in $\RR^{k+1}$ by
\begin{equation}\label{rnp1} R^{(n)} := \left( n^{-1/2}\,\mathbf{A}_{n,k}^\intercal  \xi^{(n)} , \,\,\frac{1}{n}\sum_{i=1}^n \mathbf{r}(\xi_i) \right), \qquad n \in \NN. 
\end{equation}
First note that the law of $\mathbf{A}_{n,k}$ is invariant to orthogonal transformation and independent of $\xi^{(n)}$, hence $\mathbf{A}_{n,k}^\intercal \tfrac{\xi^{(n)}}{\|\xi^{(n)}\|_2} \eqdist  	\mathbf{A}_{n,k}(1, \cdot)$ 
and $\mathbf{A}_{n,k}^\intercal \tfrac{\xi^{(n)}}{\|\xi^{(n)}\|_2}$ is independent of $\xi^{(n)}$; we refer to \cite[Lemma 6.3]{GanKimRam17} for the proof of the simpler case when $k = 1$. As a consequence,
\begin{equation*} R^{(n)} \eqdist \left(\frac{1}{n}\sqrt{n} \mathbf{A}_{n,k}(1,\cdot) \|\xi^{(n)}\|_2, \frac{1}{n} \sum_{i=1}^n \mathbf{r}(\xi_i)\right). 
\end{equation*}
Define the $\RR^{k+2}$-valued sequence of random variables,
\begin{equation*} S^{(n)} := \left( \mathbf{A}_{n,k}(1,\cdot), \frac{1}{n}\|\xi^{(n)}\|_2^2, \frac{1}{n} \sum_{i=1}^n \mathbf{r}(\xi_i)\right), \quad n\in\NN.
\end{equation*}
Since by part (iv) of Assumption \ref{as-quenched}, the domain of $\newlambda$, the log mgf of $(\xi_1^2, \mathbf{r}(\xi_1))$,
contains a  neighborhood of the origin,
  by Cram\'er's theorem  
$\{\left(\|\xi^{(n)}\|_2^2, \frac{1}{n} \sum_{i=1}^n \mathbf{r}(\xi_i)\right)\}_{n \in \NN}$ satisfies an LDP
in $\RR^2$ with the convex GRF $\newlambda^*$, equal  to the 
Legendre transform of $\newlambda$.

 The independence of $\mathbf{A}_{n,k}$ from $\{\xi_j\}_{j\in \NN}$, along with 
 Theorem 3.4 of \cite{BarGamLozRou10} (applied to the case of $p=2$ therein, with their canonically projected $X^{(k)}$ equivalent to our $\mathbf{A}_{n,k}(1, \cdot)$) then implies that the sequence $\{S^{(n)}\}_{n\in\NN}$
 satisfies an LDP with the convex GRF $J:\RR^{k+2}\ra[0,\infty]$ defined by
\begin{equation*} J(a,b,c) := -\tfrac{1}{2} \log(1-\|a\|_2^2) + \widehat{J}(b,c),
\end{equation*}
for $a\in \RR^k$ such that $\|a\|_2^2 <1$ and $b,c\in \RR$. Then, by the contraction principle, $\{R^{(n)}\}_{n\in\NN}$ satisfies an LDP with the GRF $J_R:\RR^{k+1}\ra \RR$ defined as follows: 
\begin{equation*} J_R(x,z) := \inf_{y\in \RR: y > \|x\|^2} J(xy^{-1/2}, y, z),   \quad x \in \RR^k, z \ge 0. 
\end{equation*}
Note that $J_R$ is convex due to \cite[Lemma 6.2]{GanKimRam17} and Theorem 5.3 of \cite{Roc70}.

Let $\Phi$ be as in \eqref{eq-phivar} and $T$ be as in Assumption \ref{as-quenched}(ii). 
Fix $t_1 \in \RR^k$, $t_2< T$, and let $t = (t_1, t_2)$. For $0 < \newkappa < \infty$, 
\begin{equation}\label{repnrl}
\EE [e^{\newkappa n \langle t, R^{(n)} \rangle}] =\EE\left[ e^{\newkappa n \Psi_{\mathsf{L}_{n,k}}(t_1,t_2)}\right].
\end{equation}
Since $t_2 < T$,  due to Lemma \ref{lem-bound} and the observation \eqref{repnrl}, we see that $\limsup_{n\ra\infty}\frac{1}{n}\log \EE [e^{\newkappa n \langle t, R^{(n)} \rangle}]$ is finite. By Varadhan's lemma \cite[Theorem 4.3.1]{DemZeiBook} and the finiteness condition (4.3.3) therein applied to any $\newkappa > 1$, we find that 
\begin{equation}\label{eqeq}
 \Phi(t_1,t_2) = \lim_{n\ra\infty}\frac{1}{n}\log \EE [e^{ n \langle t, R^{(n)} \rangle}] = \sup_{\tau_1 \in \RR^k,\tau_2 \in \RR} \{\langle t_1, \tau_1\rangle + t_2\tau_2 - J_R(\tau_1,\tau_2)\}. 
\end{equation}
where the first equality above holds by Lemma \ref{lem-lmgrep} on  observing that $\Phi_n$ in \eqref{phin} is the scaled log mgf of $R_n$.

Now fix $t_1 \in \RR^k$ and $t_2 \ge T$. 
We  claim that \eqref{eqeq} continues to hold,  but now with both sides equal to infinity. 
The fact that $\Phi(t_1,t_2) = \infty$ follows from the definition \eqref{eq-phivar} of $\Phi$ and the observation that
$\Phi_\nu(t_1,t_2) = \infty$ when $t_2 \ge T$ (by Remark \ref{rem-quenched}).  
To show that the term  on the right-hand side   of \eqref{eqeq} is also equal to infinity, for $s_2 \in \RR$, define $\widetilde{\Lambda}(s_2) := \Lambda(0, s_2)$. Note that $\widetilde{\Lambda}$ is the log mgf of $\mathbf{r}(\xi_1)$. Due to Assumption \ref{as-quenched}(iv),  the domain of $\widetilde{\Lambda}$ contains a non-empty neighborhood around 0, hence by Cram\'er's theorem
the sequence $\{\frac{1}{n}\sum_{i=1}^n\mathbf{r}(\xi_i)\}_{n\in\NN}$ satisfies an LDP in $\RR$ with GRF $\widetilde{\Lambda}^*$. However, due to the contraction principle and the continuity of the projection map, we also know that $\widetilde{\Lambda}^*(\tau_2) = \inf_{\tau_1\in \RR^k} J_R(\tau_1, \tau_2)$ for all $\tau_2 \in \RR$. Note that this infimum is attained at some $\tau_1^*\in\RR^k$ because, as a good rate function, $J_R$ is lower semicontinuous with compact level sets. Therefore, on the right-hand side of \eqref{eqeq}, if $t_2 \ge T$, then
\begin{align*}
\sup_{\tau_1 \in \RR^k, \tau_2 \in \RR} \{ \langle t_1, \tau_1\rangle + t_2 \tau_2 - J_R(\tau_1, \tau_2)\} &\ge \sup_{\tau_2 \in \RR} \{ \langle t_1, \tau_1^*\rangle + t_2 \tau_2 - J_R(\tau_1^*, \tau_2)\} = \langle t_1, \tau_1^*\rangle +  \widetilde{\Lambda}(t_2) = \Lambda(0, t_2) = \infty,
\end{align*}
where the first equality used the identity $(\widetilde{\Lambda}^*)^* = \widetilde{\Lambda}.$ Hence,   \eqref{eqeq} holds for \emph{all} $t_1 \in \RR^k$ and $t_2\in \RR$.

Note that \eqref{eqeq} shows that $\Phi = J_R^*$.  
Due to the convexity of $J_R$ and Legendre duality (see, e.g., \cite[Lemma 4.5.8]{DemZeiBook}), we have that $J_R = \Phi^*$. Hence, applying the contraction principle to the LDP for $\{R^{(n)}\}_{n\in\NN}$, we find that the annealed rate function $\mathcal{J}^{\sf{an}}$ of Theorem \ref{th-annealed} can be written as
\begin{align*}
\mathcal{J}^{\sf{an}}(x) &= \inf_{\tau_1\in\RR^k, \tau_2 \in \RR : \tau_1 \rho(\tau_2) = x} \Phi^*(\tau_1,\tau_2)\\	 
&= \inf_{\tau_1\in\RR^k, \tau_2 \in \RR : \tau_1 \rho(\tau_2) = x}  \inf_{\nu \in \cP(\RR^k)} \{ \Psi_\nu^*(\tau_1,\tau_2) + \mathbb{H}_k(\nu)\}\\  
&=  \inf_{\nu \in \cP(\RR^k)} \inf_{\tau_1\in\RR^k, \tau_2 \in \RR : \tau_1 \rho(\tau_2) = x}  \{ \Psi_\nu^*(\tau_1,\tau_2) + \mathbb{H}_k(\nu)\}\\	 
&= \inf_{\nu\in \cP(\RR^k)} \{ \mathcal{J}_\nu^{\sf{qu}}(x) + \mathbb{H}_k(\nu) \},
\end{align*} 
where the second equality invokes \eqref{phistarvar} of Lemma \ref{lem-phistarvar}, and the last equality relies on \eqref{def-jqu}. 
This completes the proof of Theorem \ref{th-var}.  
\end{proof}

\vskip 8mm
{ 
\parskip 0mm
\noindent Division of Applied Mathematics \\
\noindent Brown University \\
\noindent steven\_kim@alumni.brown.edu\\
\noindent kavita\_ramanan@brown.edu
}

\end{document}